\documentclass{article}
\usepackage{amssymb,amsmath,amssymb,amsthm,graphicx,url}

\newcommand{\prn}[1]{\left(#1\right)}
\newcommand{\brk}[1]{\left[#1\right]}
\newcommand{\BRK}[1]{\left\{#1\right\}}
\newcommand{\abs}[1]{\left|#1\right|}
\newcommand{\ud}[1]{\, \mathrm{d}#1}
\newcommand{\pd}[2]{\frac{\partial#1}{\partial#2}}
\newcommand{\sign}{{\rm sign}}
\renewcommand{\S}{\mathcal{S}}
\theoremstyle{plain}
\newtheorem{thm}{Theorem}
\newtheorem{lem}{Lemma}
\theoremstyle{definition}

\theoremstyle{remark}
\newtheorem{rem}{Remark}

\begin{document}
\title{A rarefaction-tracking method for hyperbolic conservation laws}
\author{Yossi Farjoun \and Benjamin Seibold
\and \\
Department of Mathematics \\
Massachusetts Institute of Technology \\
77~Massachusetts Avenue, Cambridge MA 02139}
\date{}
\maketitle

\begin{abstract}
We present a numerical method for scalar conservation laws in one space dimension. The
solution is approximated by local similarity solutions. 
While many commonly used approaches are based on shocks, the presented method uses rarefaction
and compression waves. 
The solution is
represented by particles that carry function values and move according to the method of
characteristics. Between two neighboring particles, an interpolation is defined by an
analytical similarity solution of the conservation law. An interaction of particles
represents a collision of characteristics. The resulting shock is resolved by merging
particles so that the total area under the function is conserved.
The method is variation diminishing, nevertheless, it has no numerical dissipation away
from shocks. Although shocks are not explicitly tracked, they can be located accurately.
We present numerical examples, and outline specific applications and extensions of the
approach.
\end{abstract}

\section{Introduction}
Hyperbolic conservation laws are important models for the evolution of continuum quantities,
and their efficient numerical approximation still involves challenges. A fundamental
property is the occurrence of shocks. These appear when parts of the solution that move
at different characteristic velocities collide, forming a traveling discontinuity.
A sharp and accurate resolution of shocks, without creating spurious oscillations, is one
crucial challenge in numerical methods.
Traditional finite difference \cite{LaxWendroff1960} or finite volume \cite{Godunov1959}
methods operate on a fixed grid. Straightforward linear numerical schemes are either of
low accuracy, or create overshoots near shocks \cite{Godunov1959}. These overshoots can
be avoided by considering more complex nonlinear schemes, such as finite volume methods
with limiters \cite{VanLeer1974}, and ENO \cite{HartenEngquistOsherChakravarthy1987}
or WENO \cite{LiuOsherChan1994} schemes. These difficulties show up in simple problems,
such as scalar, one dimensional conservation laws, even for linear advection. One
interpretation of this challenge is that fixed grid methods attempt to represent
advection by a change in function values.
However, since advection moves the solution along the $x-$axis, it is natural to 
move the position of the computational nodes instead.  
Two popular approaches that employ this either track shocks explicitly, or trace characteristics.
One of the first shock-tracking methods is Godunov's scheme \cite{Godunov1959}.
An initially piecewise constant function is evolved analytically by solving the local
Riemann problems exactly. This works well until two shocks (or rarefactions) interact.
Godunov's scheme avoids the resolution of this interaction by remeshing:
the solution is averaged over cells, and a new piecewise constant solution is
reconstructed. In this sense, Godunov's scheme is actually a fixed-grid finite volume
method, and thus incurs the problems described above.
The idea of tracing characteristics has been formulated by Courant, Issacson and Rees in
the CIR method \cite{CourantIsaacsonRees1952}. The solution of a conservation law is given
by a simple ordinary differential equation along a characteristic curve. The CIR method
updates function values at grid points by following characteristic curves. Function values
between grid points are defined by an interpolation, which is a remeshing step. Therefore,
the CIR method is a fixed grid approach as well. In the simplest case of using linear
interpolation between grid points, a simple non-conservative upwind scheme is obtained.

Most popular numerical schemes for conservation laws are fixed grid approaches.
Advantages include simple data structures and an easy generalization to higher space
dimensions using dimensional splitting. While these methods can be derived from
similarity solutions (finite volumes) or the method of characteristics (CIR), they
require a global remeshing step, which introduces numerical errors, numerical dissipation
and dispersion, and a decay in entropy even away from shocks. More recent approaches
attempt to avoid global remeshing, or at least weaken its negative impact.
An example of such an approach is front tracking \cite{HoldenHoldenHeghKrohn1988}.
Shocks are tracked explicitly, but unlike with Godunov-type schemes, no cells exist and no
averaging is performed. Instead, the interaction of the local similarity solutions is
resolved where necessary. The difficulty of resolving interactions of rarefactions with
shocks is avoided by approximating the flux function by a piecewise linear function.
Thus, only discontinuities have to be tracked, and rarefactions are approximated by step
functions.
Analogously, the method of characteristics can be used wherever characteristic curves do
not intersect. Lagrangian particle methods are based on this idea. Particles follow the
characteristics, and only when they come too close is their interaction resolved locally.
Classical approaches, such as smoothed particle hydrodynamics \cite{Monaghan2005},
consider a pressure that prevents particle collisions. While such approaches work well
in smooth regions, a stable treatment of shocks is challenging. In addition, hybrid
methods exist, such as the self-adjusting grid method \cite{HartenHyman1983}, which
operates on a fixed grid, and, in addition, tracks up to one shock particle in each cell.

The method presented in the current paper can be interpreted as a combination of some of
the aforementioned ideas. Characteristic curves are traced using Lagrangian particles.
In addition, between neighboring particles, a similarity solution is evolved exactly.
While front-tracking methods trace shocks, the current method uses rarefaction and
compression waves to represent the solution. Due to this analogy, the approach can be
called \emph{rarefaction tracking}. The local similarity solutions give rise to a
natural definition of area between two particles. This, in turn, admits a merging
strategy of colliding particles that conserves area exactly. Thus, shocks are
represented by compression waves that are removed upon breaking. Due to the correct
evolution of area, correct shock speeds are obtained.
The method was originally motivated in the context of Lagrangian particle
methods \cite{FarjounSeibold2009_1}. Well-posedness, accuracy and
convergence have been shown and investigated in \cite{FarjounSeibold2009_2}. 
This paper focuses on the connection to  local similarity
solutions. In addition, applications for which the
approach has advantages over fixed-grid approaches are outlined. 
The method, as it is presented here, is suited for scalar conservation
laws, since for this kind of equations similarity solutions are
straightforward to construct. 
Also, only one space dimension is currently considered, since in this case a shock
can only happen if particles collide.
However, extensions to balance laws with source terms are presented,
and other possible generalizations are outlined. 

In Sect.~\ref{sec:particle_method}, we outline the general characteristic particle
approach. In Sect.~\ref{sec:interpolation_similarity_waves}, similarity solutions between
particles are constructed. These give rise to an exact expression for area, which is the
basis for particle management, as explained in Sect.~\ref{sec:particle_management}.
Properties of the approach as presented in Sect.~\ref{sec:properties},
and fundamental extensions (for simple sources, and flux functions with
one inflection point) are outlined in Sect.~\ref{sec:extensions}.
In Sect.~\ref{sec:numerical_examples}, results of computations with the presented method
are given. We consider an ``academic'' flux function, and two physical models.
Possible applications of the method and possibilities for further
extensions are presented in Sect.~\ref{sec:further_applications}.

\section{The Particle Method}
\label{sec:particle_method}
We consider a one dimensional scalar conservation law
\begin{equation}
u_t+f(u)_x = 0, \quad u(x,0) = u_0(x)
\label{eq:conservation_law}
\end{equation}
with $f'$ continuous. In addition, until we explicitly relax the condition, we assume
that on the entire interval of considered function values $[\min_x u(x),\max_x u(x)]$,
$f$ is either convex or concave. The characteristic equations \cite{Evans1998}
\begin{equation}
\begin{cases}
\dot x = f'(u) \\
\dot u = 0
\end{cases}
\label{eq:characteristic_equation}
\end{equation}
yield the solution (while it is smooth) forward in time:
at each point $(x_0,u_0(x_0))$ a characteristic curve $x(t) = x_0+f'(u_0(x_0))\, t$ starts,
carrying the function value $u(x(t),t) = u_0(x_0)$.
The method of characteristics gives rise to a particle method:
\begin{enumerate}
\item
Sample the initial function $u_0(x)$ by a finite number of points $(x_i,u_i)$.
The aspect of sampling the initial conditions is addressed in
Sect.~\ref{subsec:sampling_initial_data}.
\item
The solution at time $t>0$ is given by the points $(x_i+f'(u_i)t,u_i)$.
At these points, the solution is represented exactly, without any approximation error.
\end{enumerate}
This basic approach incurs two fundamental limitations, which need to be remedied
before it can become a useful method:
\begin{itemize}
\item
The solution is only known at the particle positions. If neighboring particles
depart, large regions of unspecified function value may arise.
The presented approach remedies this problem by defining an interpolation between
particles. This allows an easy insertion of new particles into gaps wherever desired.
\item
A particle may overtake its neighbor, corresponding to the intersection of
characteristic curves. The correct weak solution of the conservation
law \eqref{eq:conservation_law} develops a shock, i.e.~a moving discontinuity, into
which the characteristic curves keep running \cite{Evans1998}.
In the presented approach, particles are merged upon collision. An interpolation
between particles allows the merging of particles in such a way that area is
conserved.
\end{itemize}
The insertion and merging of particles is called \emph{particle management}.
It is described in detail in Sect.~\ref{sec:particle_management}.
Particle management is the key ingredient that enables us to use the method of
characteristics beyond the first occurrence of shocks.

Given the solution at time $t$, represented by particles $(x_i(t),u_i(t))$, the
time of the next particle collision is $t+\varDelta t_{\text s}$, where
\begin{equation}
\varDelta t_i = -\frac{x_{i+1}-x_i}{f'(u_{i+1})-f'(u_i)}
\label{eq:intersection_time}
\end{equation}
is the time until the collision between particles $i$ and $i+1$, and
\begin{equation}
\varDelta t_{\text s} = \min\BRK{\BRK{\varDelta t_i | \varDelta t_i\ge 0}\cup\infty}
\label{eq:time_step}
\end{equation}
is the time until the next collision in the future. We can step forward  by up to
$\varDelta t_{\text s}$ without risking a collision. 
After $\varDelta t_{\text s}$, at least one particle will share its
position with another. To proceed further, we need to perform particle management.
We insert new particles into gaps between particles of distance larger than a maximum
distance $d_\text{max}$. Then, each pair of particles sharing their position is merged.
Optionally, one can perform a merging step for all particles closer than a minimum
distance $d_\text{min}$. Choices for the maximum and minimum distance are described in
Sect.~\ref{sec:particle_management}.

\section{Interpolation by Similarity Waves}
\label{sec:interpolation_similarity_waves}
Consider two neighboring particles $(x_1(t),u_1)$, $(x_2(t),u_2)$, with $x_1<x_2$.
If $f'(u_1)>f'(u_2)$, the two particles meet in the future, at time $t+\varDelta t_1$,
where $\varDelta t_1>0$ is given by \eqref{eq:intersection_time}.
Let us assume that between the two particles, the solution is continuous at all times
before $t+\varDelta t_1$. This implies that the meeting of the particles is the first
occurrence of a shock, and at this time the solution is a discontinuity at position
$x_{\text{sh}} = x_1+f'(u_1)\varDelta t_1$.
Any point $(x_{\text{sh}},u)$, $u\in [u_1,u_2]$ on the shock can be traced backwards
along its characteristic to \mbox{$(x_{\text{sh}}-f'(u)\varDelta t_1,u)$}.
As shown in Fig.~\ref{fig:def_interp}, this defines the interpolation uniquely as
\begin{equation}
x(u) = x_1+f'(u_1)\varDelta t_1-f'(u)\varDelta t_1
= x_1+\frac{f'(u)-f'(u_1)}{f'(u_2)-f'(u_1)}(x_2-x_1) \;.
\label{eq:interpolation_function}
\end{equation}
In the case of departing particles, i.e.~$f'(u_1)<f'(u_2)$, the above construction is
exactly reversed in space and time. We assume that the solution between the particles
came from a discontinuity at some time in the past. Consequently,
formula \eqref{eq:interpolation_function} also describes this case.
\begin{thm}
\label{thm:particle_movement_solution}
The interpolation \eqref{eq:interpolation_function} between two particles is a solution
of the conservation law \eqref{eq:conservation_law}. Consequently, for multiple
particles, the function defined piecewise by the
interpolation \eqref{eq:interpolation_function} is a classical (i.e.~continuous)
solution to the conservation law \eqref{eq:conservation_law}, until particles collide.
\end{thm}
\begin{proof}
Using that \mbox{$\dot x_i(t)=f'(u_i)$} for~$i=1,2$ (from
\eqref{eq:characteristic_equation}) one obtains
\begin{align*}
\pd{x}{t}(u,t) &= \dot x_1
+\frac{f'(u)-f'(u_1)}{f'(u_2)-f'(u_1)}(\dot x_2-\dot x_1)\\
&= f'(u_1)+\frac{f'(u)-f'(u_1)}{f'(u_2)-f'(u_1)}(f'(u_2)-f'(u_1))
= f'(u) \;.
\end{align*}
Thus every point on the interpolation $u(x,t)$ satisfies the characteristic
equation \eqref{eq:characteristic_equation}. At the particles, the solution possesses
kinks. These can be interpreted as shocks of height zero, and thus satisfy the
Rankine-Hugoniot conditions \cite{Evans1998}. Hence, the full piecewise wave solution
is a weak solution of
\eqref{eq:conservation_law}.
\end{proof}
The interpolation \eqref{eq:interpolation_function} defines either a rarefaction wave
(if $f'(u_1)<f'(u_2)$), or a compression wave (if $f'(u_1)>f'(u_2)$) between two
particles. In the case $u_1=u_2$, a constant interpolant is defined. Hence, at all
times the solution is approximated by similarity waves.
\begin{rem}
That \eqref{eq:interpolation_function} defines the interpolant as a function
$x(u)$ is, in fact, an advantage, since at a discontinuity $x_1 = x_2$, characteristics
for all intermediate values $u$ are defined. Thus, rarefaction fans arise naturally
from discontinuities if $f'(u_1)<f'(u_2)$. If $f$ has no inflection points between
$u_1$ and $u_2$, the inverse function $u(x)$ exists. For plotting purposes we plot
$x(u)$ instead of inverting the function.
\end{rem}

\subsection{Evolution of Area}
The interpolation defines an area between particles.
As shown in \cite{FarjounSeibold2009_2}, the integration of
\eqref{eq:interpolation_function} yields
\begin{equation}
\int_{x_1(t)}^{x_2(t)} u(x,t)\ud{x} = (x_2(t)-x_1(t))\,a(u_1,u_2) \;,
\label{eq:area}
\end{equation}
where $a(u_1,u_2)$ is the nonlinear average function
\begin{equation}
a(u_1,u_2)
= \frac{\brk{f'(u)u-f(u)}_{u_1}^{u_2}}{\brk{f'(u)}_{u_1}^{u_2}}
= \frac{\int_{u_1}^{u_2}f''(u)\,u\ud{u}}{\int_{u_1}^{u_2}f''(u)\ud{u}}\;.
\label{eq:nonlinear_average}
\end{equation}
The integral form shows that $a$ is indeed an average of $u$, weighted by $f''$.
In \cite{FarjounSeibold2009_2} it is shown that $a(u_1,u_2)$ is symmetric, bounded
by its arguments, strictly increasing in both arguments, and continuous at $u_1=u_2$.

The area under the interpolant can also be derived from the conservation law
\eqref{eq:conservation_law} by observing that the change of area between two
characteristic particles $(x_1(t),u_1)$ and $(x_2(t),u_2)$ is given by
\begin{equation}
\begin{split}
\frac{d}{dt}\int_{x_1(t)}^{x_2(t)}u(x,t)\ud{x}
&= \prn{f'(u_2)u_2-f(u_2)}-\prn{f'(u_1)u_1-f(u_1)} \\
&= F(u_2)-F(u_1) = \brk{F}_{u_1}^{u_2}\;,
\end{split}
\label{eq:area_deriv}
\end{equation}
where $F = f'(u)u-f$ is the Legendre transform of the flux function $f$.
Since $u$ is conserved along the characteristics, $F$ is conserved as well.
Thus, the area between particles changes linearly in time, as does their distance
\begin{equation}
\frac{d}{dt}(x_2(t)-x_1(t)) = \dot{x}_2(t)-\dot{x}_1(t) = f'(u_2)-f'(u_1)
= \brk{f'(u)}_{u_1}^{u_2}\;.
\label{eq:distance_deriv}
\end{equation}
As before, these results require that the solution remains continuous until the particles
collide. At the time of collision $t_0$, both the distance and the area vanish.
Thus \eqref{eq:area_deriv} and \eqref{eq:distance_deriv} integrate to
\begin{align*}
\int_{x_1(t)}^{x_2(t)} u(x,t)\ud{x} &= (t-t_0) \cdot \brk{F(u)}_{u_1}^{u_2} \;, \text{ and }\\
x_2(t)-x_1(t) &= (t-t_0) \cdot \brk{f'(u)}_{u_1}^{u_2}\;,
\end{align*}
which lead to \eqref{eq:area}.
\begin{rem}
The definition of area between particles gives rise to an exactly conservative
resolution of particle interaction. This poses an interesting analogy to finite volume
methods. While these are derived by the flux $f$ through fixed boundaries
\begin{equation*}
\frac{d}{dt}\int_{x_1}^{x_2}u(x,t)\ud{x} = -\brk{f}_{u_1}^{u_2}\;,
\end{equation*}
the presented method is based on the Lagrangian flux $F$ through moving boundaries,
as given by \eqref{eq:area_deriv}.
\end{rem}

\begin{figure}
\centering
\begin{minipage}[t]{.49\textwidth}
\centering
\includegraphics[width=0.99\textwidth]{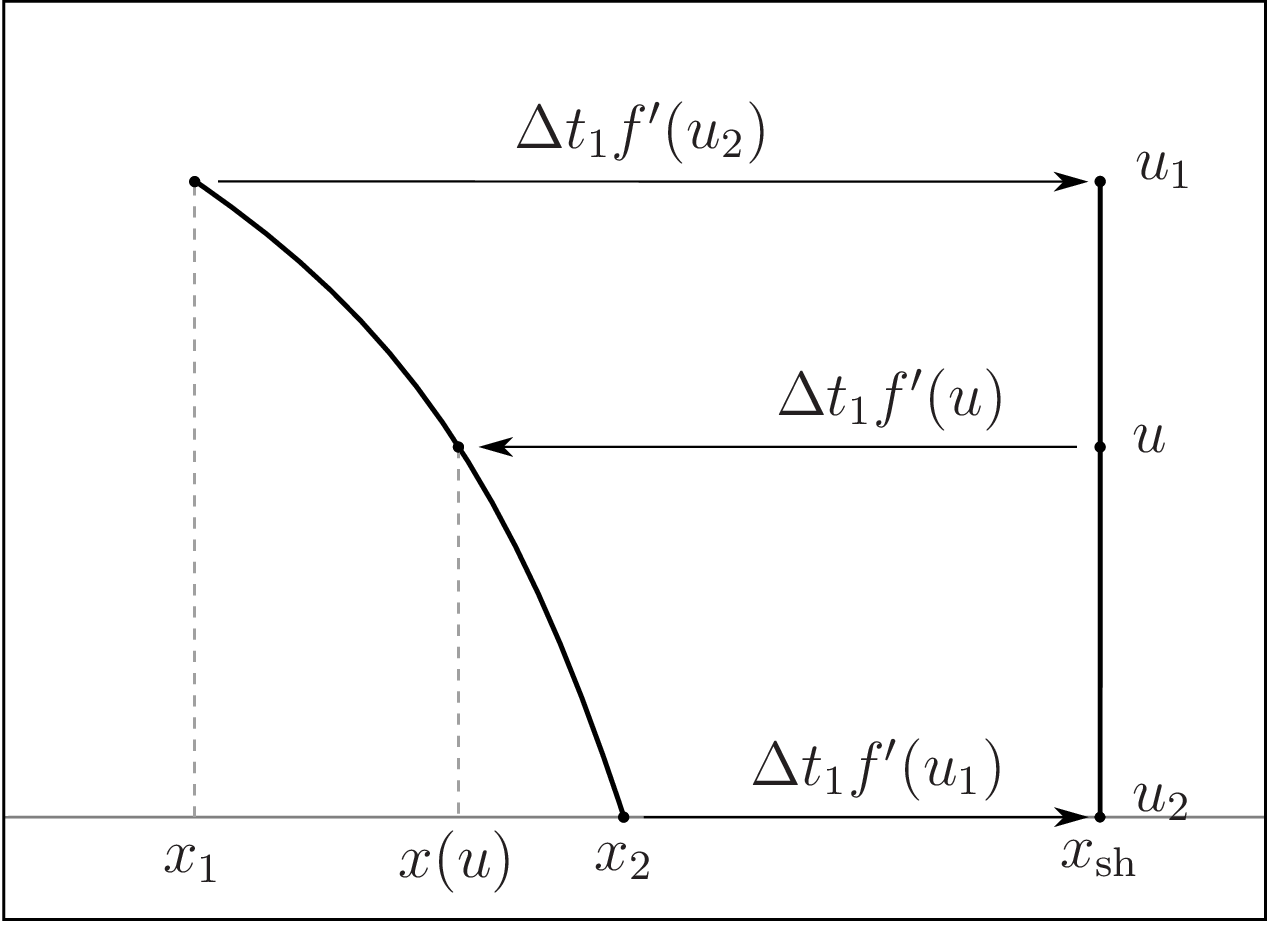}
\caption{Definition of the interpolation. After time $\Delta t_1$, particles 1\&2 form a
shock. From this $x(u)$ is found.}
\label{fig:def_interp}
\end{minipage}
\hfill
\begin{minipage}[t]{.49\textwidth}
\centering
\includegraphics[width=0.99\textwidth]{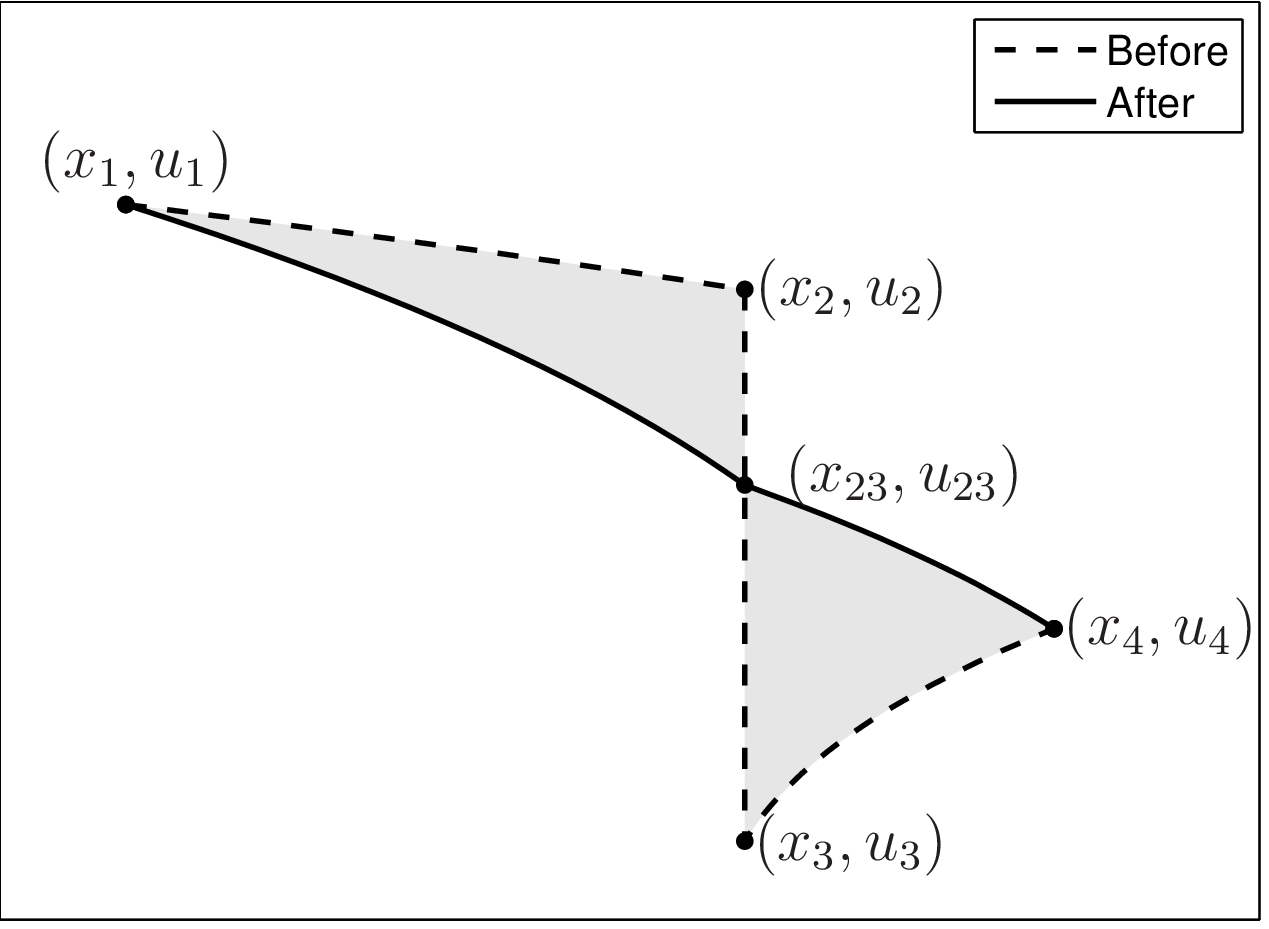}
\caption{Merging particles 2\&3. The values $(x_{23}, u_{23})$ are found so that the area
of the two ``triangles'' is equal.}
\label{fig:merging}
\end{minipage}
\end{figure}

\subsection{Sampling of the Initial Conditions}
\label{subsec:sampling_initial_data}
Given particle positions and function values, the interpolation
\eqref{eq:interpolation_function} defines a function everywhere.
If the initial conditions $u_0(x)$ can be represented exactly by finitely many particles
and the interpolation \eqref{eq:interpolation_function} in between, then the characteristic
evolution yields an \emph{exact} solution of the conservation law. In general, the initial
conditions must be approximated by the piecewise interpolation
\eqref{eq:interpolation_function}. This approximation step is similar to finite element
approaches. The selection of initial particle positions $x_i$ corresponds to the
construction of a mesh. The function values $u_i$ can be chosen to minimize the
approximation error between the exact initial function and the numerical approximation.
However, there are two aspects of the initial sampling that are very specific to
conservation laws and the numerical approach:
\begin{itemize}
\item
The initial function $u_0(x)$ may involve discontinuities. We assume only finitely many
jumps, and that they are known. Each jump can be represented exactly by two particles at
the same position (but with different function values). Having this option is one
definite advantage of particle methods over Eulerian approaches.
\item
\emph{Area} is of great importance for conservation laws. The interpolation yields an exact
knowledge of area in the numerical approximation, and the approximation step of the
initial function can incorporate this aspect, and chose an initial sampling that respects
the area of $u_0(x)$.
\end{itemize}
A simple approach is to represent jumps exactly, and then sample positions $x_i$
equidistantly between jumps while choosing function values $u_i = u_0(x_i)$. Obviously,
better sampling strategies are possible, but this simple approach often yields
satisfying results. In \cite{FarjounSeibold2009_2}, we prove that with this simple
sampling strategy, the piecewise interpolation \eqref{eq:interpolation_function} is a
second order accurate approximation to any $C^2$ function with
$|f''(u_0(x))|\ge C>0$, i.e.~the initial function does not cross nor touch
any inflection point of $f$. For non-convex flux functions, as outlined in
Sect.~\ref{subsec:inflection_points}, the second order accuracy is, in general, lost
when $u_0$ crosses an inflection point of $f$.

\section{Particle Management}
\label{sec:particle_management}
The time evolution of the conservation law \eqref{eq:conservation_law} is described by
the characteristic movement of the particles \eqref{eq:characteristic_equation}.
Particle management is an ``instantaneous'' operation (i.e.~happening at constant time)
that allows the method to continue an evolution forward in time. It is designed to
conserve area: The function value of an inserted or merged particle is chosen so that
the total area is unchanged by the operation.

Consider four neighboring particles located at
\mbox{$x_1<x_2\le x_3<x_4$}\footnote{If more than two particles are at one position ($x$),
all but the one with the smallest value ($u$) and the one with the largest value ($u$)
are removed immediately.}
with associated function values $u_1$, $u_2$, $u_3$, $u_4$, and $u_2\neq u_3$.
There are two possible cases that require particle management between $x_2$ and $x_3$:
\begin{itemize}
\item \textbf{Insertion:}
If two particles deviate from each other ($f'(u_2)<f'(u_3)$), their distance may become
unacceptably large: $x_3-x_2\ge d_\text{max}$. In this case, we insert a new particle
$(x_{23},u_{23})$ with $x_2<x_{23}<x_3$, such that the area is preserved:
\begin{equation}
(x_{23}-x_2)\,a(u_2,u_{23})+(x_3-x_{23})\,a(u_{23},u_3) = (x_3-x_2)\,a(u_2,u_3) \;.
\label{eq:area_cond_insert}
\end{equation}
One can, for example, set $x_{23}=\frac{x_2+x_3}{2}$ and find $u_{23}$
by \eqref{eq:area_cond_insert}, or set $u_{23}=\frac{u_2+u_3}{2}$ and find $x_{23}$
by \eqref{eq:area_cond_insert}.
\item \textbf{Merging:}
If two particles approach each other ($f'(u_2)>f'(u_3)$), they will eventually collide:
$x_3-x_2\le d_\text{min}$. In this case, we replace both with a new particle
$(x_{23},u_{23})$ at $x_{23}=\frac{x_2+x_3}{2}$, with $u_{23}$ chosen, such that the
area is preserved:
\begin{equation}
\begin{split}
(x_{23}-x_1)\,a(u_1,u_{23})+(x_4-x_{23})\,a(u_{23},u_4)& \\
= (x_2\!-\!x_1)\,a(u_1,u_2)+(x_3\!-\!x_2)\,a(u_2,u_3)&+(x_4\!-\!x_3)\,a(u_3,u_4)\;.
\end{split}
\label{eq:area_cond_merge}
\end{equation}
Figure~\ref{fig:merging} illustrates this merging step.
The merging step is extended by an entropy fix, as described in
Sect.~\ref{subsec:entropy_fix}.
\end{itemize}
The two distance parameters have very natural meanings for the method. The maximum
distance $d_\text{max}$ is the resolution of the method. Due to insertion, the distance
between particles on rarefactions is always between 
$d_\text{max}/2$ and $d_\text{max}$. The minimum distance $d_\text{min}$ is ideally
zero, therefore it is not a real parameter of the method. For stability reasons, one
can assign a small positive value. This is of interest in extensions to balance laws,
in which case the characteristic equations may have to be approximated numerically.

We observe that inserting and merging are similar in nature.
Conditions \eqref{eq:area_cond_insert} and \eqref{eq:area_cond_merge} for $u_{23}$ are
nonlinear (unless $f$ is quadratic, see Rem.~\ref{rem:quadratic_flux}).
For most cases, $u_{23}=\frac{u_2+u_3}{2}$ is a good initial guess, and the correct value
can be obtained (up to the desired precision) by a few Newton iteration steps
(or bisection, if the Newton iteration fails to converge).
It is shown in \cite{FarjounSeibold2009_2}, that the new particle value $u_{23}$ for
intersection and merging always exists, is unique, and bounded by the surrounding
particle values $u_{23}\in [u_2,u_3]$. For merging, the latter result is easy to show
when $x_2=x_3$, but it also holds when the slope
$|u_3-u_2|/|x_3-x_2|$ is larger than a given threshold (which depends on the
surrounding particle geometry $x_1,\dots,x_4$, $u_1,\dots,u_4$).
Thus, the merging step is robust with respect to small deviations in the distance of
the merged particles. This also holds for the case $x_2>x_3$, i.e.~merging can in
principle be performed \emph{after} two particles have overtaken.

\subsection{Entropy Fix for Merging}
\label{subsec:entropy_fix}
The merging of particles resolves shocks, which are moving discontinuities. A weak
formulation of the conservation law \eqref{eq:conservation_law} admits these solutions.
For some initial conditions, the weak solution concept allows for infinitely many
solutions, and an additional condition has to be imposed to single out a unique solution.
This can be done by defining a convex entropy function $\eta$, and an entropy flux $q$,
and requiring the \emph{entropy condition}
\begin{equation}
\eta(u)_t+q(u)_x\le 0
\label{eq:entropy_condition}
\end{equation}
to be satisfied. This condition is equivalent to obtaining shocks only when
characteristics run into them \cite{Evans1998}.
It is shown in \cite[Chap.~2.1]{HoldenRisebro2002} that if \eqref{eq:entropy_condition}
is satisfied for the Kru{\v z}kov entropy pair $\eta_k(u) = \abs{u-k}$,
$q_k(u) = \sign(u-k)(f(u)-f(k))$, for all $k$, then it is satisfied for any convex
entropy function.

Relation \eqref{eq:entropy_condition} implies that the total entropy
$\int\eta(u(x,t))\ud{x}$ of a solution does not increase in time.
Particle merging changes the solution locally. It is designed to conserved area
exactly. Similarly, we require that the merging step not increase the total entropy.
In \cite{FarjounSeibold2009_2}, it is shown that this property is satisfied if
shocks are reasonably well resolved.
\begin{lem}
\label{lem:entropy_merging}
Let $x_1<x_2=x_3<x_4$ be the locations of four particles, with particles
$2$ and $3$ to be merged, thus $u_2>u_3$ (if $f''>0$), respectively
$u_2<u_3$ (if $f''<0$). If the value $u_{23}$ resulting from the merge satisfies
$u_1\ge u_{23}\ge u_4$ (if $f''>0$), respectively $u_1\le u_{23}\le u_4$ (if $f''<0$),
then the Kru{\v z}kov entropy $\int\abs{u-k}\ud{x}$ does not increase due to the merge.
\end{lem}
This entropy condition is ensured by an \emph{entropy fix}:
A merge is rejected \emph{a posteriori} if the condition is not satisfied.
In this case, points are inserted half-way between the shock and the
neighboring particles, and the merge is re-attempted.
The condition is satisfied if new particles are inserted sufficiently
close to the shock, hence this approach is guaranteed to succeed.

\subsection{Shock Postprocessing}
\label{subsec:shock_postprocessing}
The particle method does not explicitly track shocks. However, shocks can be located.
Whenever particles are merged, the new particle can be marked as a \emph{shock particle}.
Thus, any shock stretches over three particles $(x_1,u_1),(x_2,u_2),(x_3,u_3)$, with the
shock particle in the middle. Before plotting or interpreting the solution, a postprocessing
step can be performed: The shock particle is replaced by a discontinuity, represented
by two particles $(\bar{x}_2,u_1),(\bar{x}_2,u_3)$, with their position $\bar{x}_2$
chosen such that area is conserved exactly. This step is harmless, since an immediate
particle merge would recover the original configuration. The numerical results
in Sect.~\ref{subsec:numerics_convergence_error} indicate that, with this postprocessing,
the resulting solution is second order accurate in $L^1$. In particular,
the shock is located with second order accuracy.

\section{Properties of the Method}
\label{sec:properties}
The presented method consists of two ingredients: time evolution by characteristic
particle movement, and particle management at stationary time. Due to the concept of
similarity solutions between particles, an exact solution of the conservation law
\eqref{eq:conservation_law} is evolved during particle movement
(see Thm.~\ref{thm:particle_movement_solution}).
\begin{thm}[TVD and entropy diminishing]
Solutions obtained by the particle method are total variation diminishing and
satisfy the entropy condition \eqref{eq:entropy_condition}. Hence, no spurious
oscillations are created, and correct entropy solutions are approximated.
\end{thm}
\begin{proof}
Due to Thm.~\ref{thm:particle_movement_solution}, the characteristic particle movement
yields an exact classical solution. Thus, both the total variation and the entropy of
the solution are constant. Particle insertion simply refines the interpolation, without
changing the solution. Particle merging is shown to yield a new value bounded by the
surrounding particle values (Sect.~\ref{sec:particle_management}), thus it is TVD.
The entropy fix (Sect.~\ref{subsec:entropy_fix}) ensures the entropy condition is
satisfied.
\end{proof}
\begin{rem}[Quadratic flux function]
\label{rem:quadratic_flux}
The method is particularly efficient for quadratic flux functions.
In this case the interpolation \eqref{eq:interpolation_function} between two
points is a straight line, and the average \eqref{eq:nonlinear_average} is the
arithmetic mean $a(u_1,u_2) = \frac{u_1+u_2}{2}$. Thus, the function values for
particle management can be computed explicitly.
\end{rem}
\begin{rem}[Computational Cost]
An interesting aspect arises in the computational cost of the method, when counting
evaluations of the flux function $f$ and its derivatives $f'$, $f''$. Particle movement
does not involve any evaluations since the characteristic velocity
of each particle $f'(u_i)$ does not change. Consider a solution on $t\in [0,1]$ with
a bounded number of shocks, to be approximated by $O(n)$ particles.
Every particle merge and insertion requires $O(1)$ evaluations.
After each time increment \eqref{eq:time_step}, $O(1)$ management steps are
required. The total number of time increment steps is $O(n)$.
Thus, the total cost is $O(n)$ evaluations, as opposed to $O(n^2)$ evaluations for
Godunov-type methods. Note that this aspect is only apparent if evaluations of
$f'$, $f''$ are expensive, since the total number of operations is still $O(n^2)$.
\end{rem}

\section{Fundamental Extensions and Generalizations}
\label{sec:extensions}
The basic method, presented in the previous sections, is formulated for scalar
one-dimensional hyperbolic conservation laws with a flux function that is convex or
concave. In this section, we present how some of these restrictions can be relaxed.
In Sect.~\ref{subsec:inflection_points}, we describe the treatment of flux functions with
one inflection point. Section~\ref{subsec:source_terms} addresses the incorporation of
simple source terms.

\subsection{Non-Convex Flux Functions}
\label{subsec:inflection_points}
The key idea in generalizing the method for flux functions that have inflection points is
to carry \emph{inflection particles}, such that between neighboring particles, the flux
function is always either convex or concave. At inflection particles, the interpolant
\eqref{eq:interpolation_function} has an infinite slope (since $f''=0$ at this point),
which poses some difficulty for a high order approximation of initial functions that cross
an inflection point \cite{FarjounSeibold2009_2}. The characteristic particle movement, and
the expressions for area and average \eqref{eq:nonlinear_average} apply as before. However,
the merging of particles when an inflection particle is involved requires a special
treatment. The standard approach, as presented in Sect.~\ref{sec:particle_management},
replaces two particles by one particle of a different function value. If an inflection
particle is involved in a collision, the merging must be done in a different way so that
an inflection particle remains.

The key idea is to move the inflection particle's position, while preserving its function
value. Depending on the configuration, one of three different merging strategies has to be
performed. Figure~\ref{fig:particle_management_inflection} shows these merging strategies
for the case of a single inflection point (we do not consider here the interaction of
multiple inflection points), with $f'''>0$ (the inflection particle is the slowest), and
particles 2 and 3 to be merged. The other cases are simple symmetries of this situation.
Each of the following strategies is attempted if the previous one failed:
\begin{enumerate}
\item
Remove particle 2 and increase $x_3$ such that area is preserved.
Accept if $x_3$ need not be increased beyond $x_4$.
\item
Remove particle 2, set $x_3=x_4$ and increase both such that area is preserved.
Accept if new value for $x_3$ and $x_4$ is smaller than  $x_5$.
\item
Remove particle 4, set $x_3=x_5$ and lower $u_2$ such that area is preserved.
\end{enumerate}
It is shown in \cite{FarjounSeibold2009_2} that one of these three options must be
accepted. Note that the final configuration may involve a new discontinuity
($x_3=x_4$ or $x_3=x_5$). However, this is not a shock, but rather a rarefaction, and the
particles will move away from each other. Consequently, these particles should \emph{not}
be merged. The presented strategy guarantees that in each merging step one particle is
removed.

\begin{figure}
\centering
\begin{minipage}[t]{.32\textwidth}
\centering
\includegraphics[width=.99\textwidth]{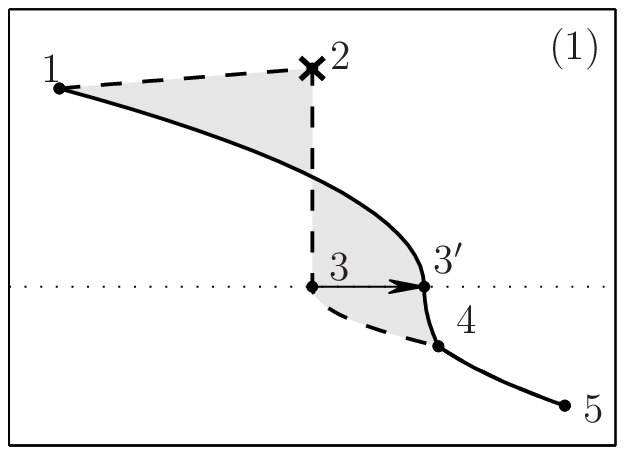}
\end{minipage}
\hfill
\begin{minipage}[t]{.32\textwidth}
\centering
\includegraphics[width=.99\textwidth]{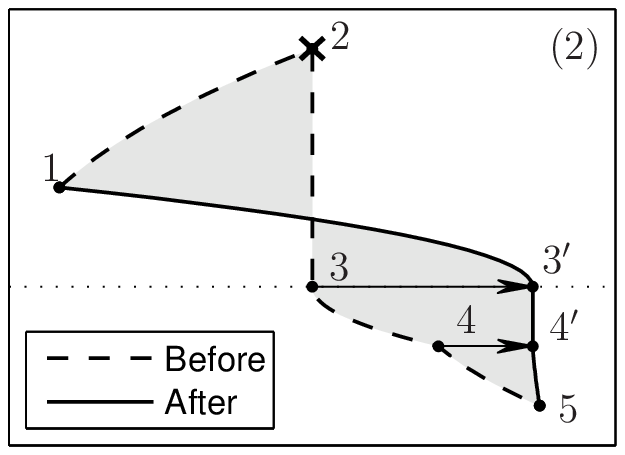}
\end{minipage}
\hfill
\begin{minipage}[t]{.32\textwidth}
\centering
\includegraphics[width=.99\textwidth]{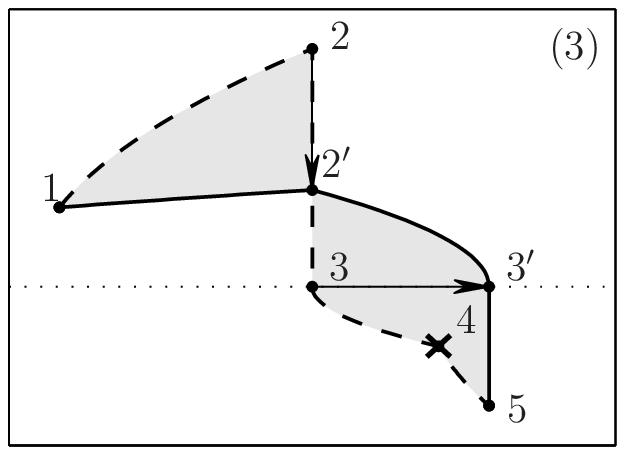}
\end{minipage}
\caption{Particle management around an inflection particle ($f''(u_3)=0$)}
\label{fig:particle_management_inflection}
\end{figure}

\subsection{Simple Source Terms}
\label{subsec:source_terms}
Source terms extend the conservation law \eqref{eq:conservation_law} by a nonzero right
hand side
\begin{equation}
u_t+f(u)_x = \S u\;,
\label{eq:conservation_source}
\end{equation}
where $\S$ is an operator on $u$.
Here we consider the case where the source term is a simple function
$\S u = g(x,u)$, as opposed to a differential or integral operator.
An example of $\S$ being an integral operator is outlined in
Sect.~\ref{subsec:source_intergral}.

If the source is a simple function $\S u = g(x,u)$, equation
\eqref{eq:conservation_source} becomes a balance law. In this case, the method of
characteristics \cite{Evans1998} applies. The presented particle method can be
generalized as follows. During the particle movement, the characteristic equations
\begin{equation}
\begin{cases}
\dot x = f'(u) \\
\dot u = g(x,u)
\end{cases}
\label{eq:method_of_characteristics_source}
\end{equation}
are solved numerically, for instance by an explicit Runge-Kutta scheme. Merging takes
place when two particles are too close. Particle management is performed according
to \eqref{eq:area_cond_insert} and \eqref{eq:area_cond_merge}, as derived for the
source-free case. This is an approximation, since the interpolation
\eqref{eq:interpolation_function} is \emph{not} a solution of the balance law
\eqref{eq:conservation_source}. If the advection dominates over the source, we expect the
approximation to be very accurate. The numerical results in
Sect.~\ref{subsec:numerics_source_terms} indicate that this approach can yield good
results even when the source and the advection term are of comparable magnitude.
The solution is found to have second order convergence, and accuracy that compares
favorably to CLAWPACK \cite{Clawpack}.

\begin{figure}
\centering
\begin{minipage}[t]{.32\textwidth}
\centering
\includegraphics[width=0.99\textwidth]{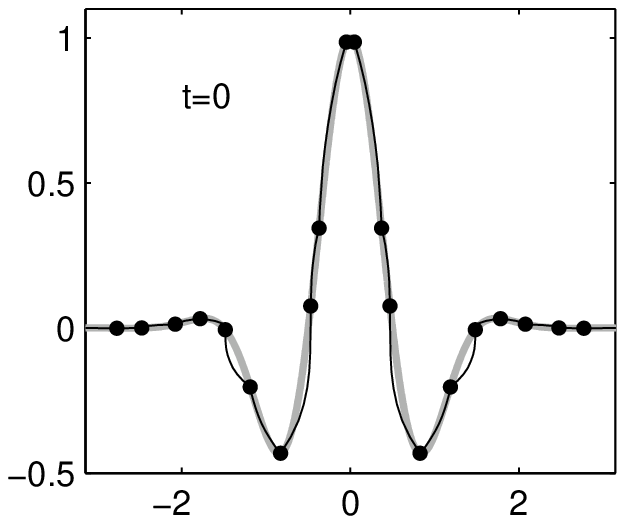}
\end{minipage}
\hfill
\begin{minipage}[t]{.32\textwidth}
\centering
\includegraphics[width=0.99\textwidth]{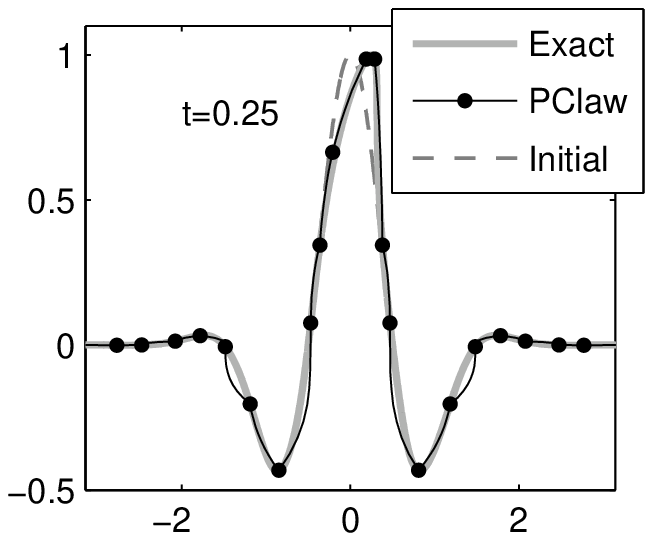}
\end{minipage}
\hfill
\begin{minipage}[t]{.32\textwidth}
\centering
\includegraphics[width=0.99\textwidth]{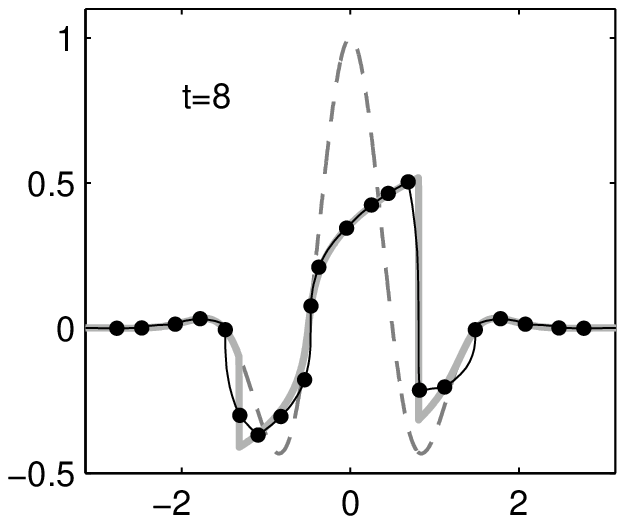}
\end{minipage}
\caption{The particle method for $f(u)=\tfrac{1}{4}u^4$ before and after shocks arise}
\label{fig:u4}
\end{figure}

\section{Numerical Examples}
\label{sec:numerical_examples}
The particle method is numerically investigated in various examples. In all cases, the
reference solution is obtained by a high resolution CLAWPACK \cite{Clawpack} computation. 
A more rigorous exposition of the order of convergence and comparisons with other methods
are presented in \cite{FarjounSeibold2009_2}.

In Sect.~\ref{subsec:numerics_convergence_error}, the evolution and the formation of shocks
for smooth initial data under a convex flux function are presented. The reduction of order
due to a shock (if no postprocessing is done) is shown, as is the effectiveness of the
postprocessing in recovering the second-order accuracy.
In Sect.~\ref{subsec:numerics_non_convex_flux}, the Buckley-Leverett equation is
considered, as an example of a non-convex flux function.
In Sect.~\ref{subsec:numerics_source_terms}, Burgers' equation with a source is simulated.
The source code and all presented examples can be found on the \texttt{particleclaw} web
page \cite{Particleclaw}.

\subsection{The Formation of Shocks}
\label{subsec:numerics_convergence_error}
Figure~\ref{fig:u4} shows the smooth initial function
$u_0(x) = \exp\!\prn{-x^2}\cos(\pi x)$, and its time evolution under the flux
function $f(u)=\frac{1}{4}u^4$. Note the curved shape of the
interpolation \eqref{eq:interpolation_function}. Initially, we sample points on the
function $u_0$. At time $t=0.25$, the solution is still smooth, thus the particles lie
exactly on the solution. At time $t=8$, shocks and rarefactions have occurred and
interacted. Although the numerical solution uses only a few points, it represents the
true solution well.

\begin{figure}
\centering
\begin{minipage}[t]{.48\textwidth}
\centering
\includegraphics[width=.99\textwidth]{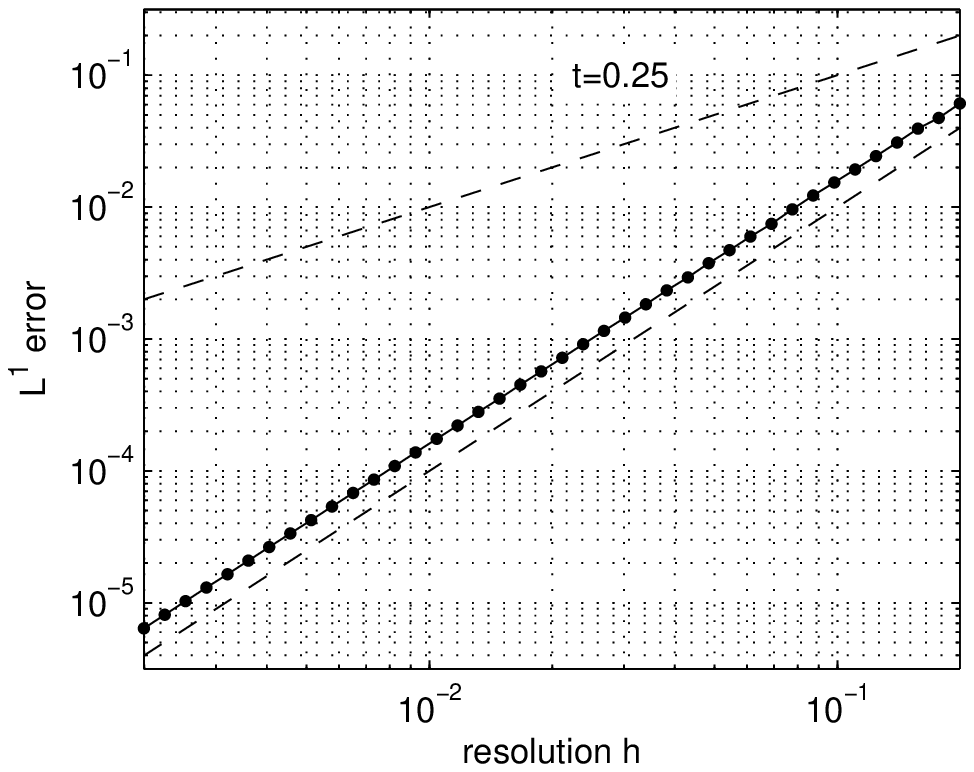}
\end{minipage}
\hfill
\begin{minipage}[t]{.48\textwidth}
\centering
\includegraphics[width=.99\textwidth]{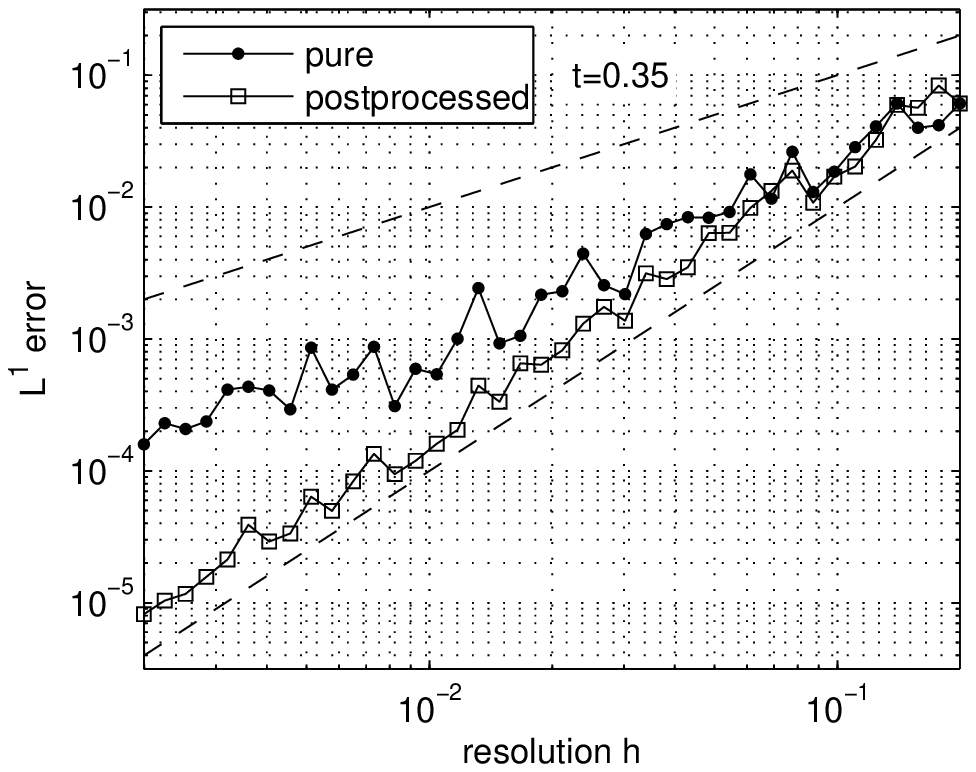}
\end{minipage}
\caption{$L^1$ Error of the particle method for $f(u)=\frac14u^4$, before and after a
shock. In the left plot the error is \emph{only} due to the initial sampling. In the
right plot the advantage of ``postprocessing'' is made apparent, as it converts the
first-order accuracy of the ``pure'' method, into second-order.}
\label{fig:error}
\end{figure}

The error for different numbers of particles is shown in Fig.~\ref{fig:error}. Before shocks
have emerged, the solution is smooth and the error is only due to the initial sampling,
and therefore second-order. After a shock has emerged, the situation is not so simple:
the exact solution has a jump discontinuity, and we measure the error according to the
$L^1$ norm of the difference between the interpolation and the ``exact'' solution.
Therefore, since we normally have just one particle near the shock, we will make an
error proportional to the height$\times$width of the shock. Since the width of the shock
is generally proportional to the resolution of the method, we get first order accuracy.
However, since we have the correct \emph{area}, the postprocessing step, explained
in Sect.~\ref{subsec:shock_postprocessing}, recovers the second order accuracy of this
solution, as is evident from the figure.

\subsection{A Non-Convex Flux Function}
\label{subsec:numerics_non_convex_flux}
As an example of a non-convex flux function, we consider the Buckley-Leverett equation,
which is a simple model for two-phase fluid flow in a porous medium
(see LeVeque \cite[Chapter 16]{LeVeque2002}). The flux function is 
\begin{equation}
f(u) = u^2/(u^2+\tfrac{1}{2}(1-u)^2)\;.
\end{equation}
We consider piecewise constant initial data with a large downward jump that crosses the
inflection point $u^* = 0.387$, and a small upward jump:
\begin{equation}
u_0(x) = \begin{cases}
1 & x\le1\\ 
0 & 1<x\le1.3\\
0.3 & 1.3<x\;.
\end{cases} 
\end{equation}
The large jump develops a shock at the bottom
and a rarefaction at the top, while the small jump develops a pure rarefaction.
Around $t=0.2$, the two similarity solutions interact, thereby lowering the separation
point between shock and rarefaction. The solution is shown in
Fig.~\ref{fig:results_inflection}.
The analysis presented in \cite{FarjounSeibold2009_2} shows that the convergence
obtained here is less than second order, due to the error on the intervals that abut the
inflection particle.

\begin{figure}
\centering
\begin{minipage}[t]{.32\textwidth}
\centering
\includegraphics[width=1.02\textwidth]{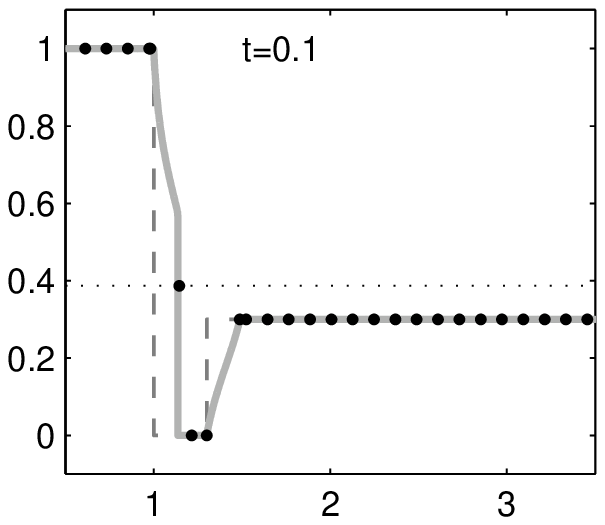}
\end{minipage}
\hfill
\begin{minipage}[t]{.32\textwidth}
\centering
\includegraphics[width=1.02\textwidth]{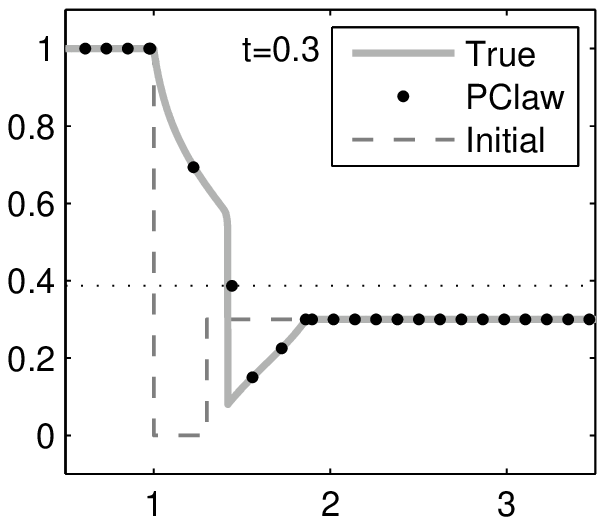}
\end{minipage}
\hfill
\begin{minipage}[t]{.32\textwidth}
\centering
\includegraphics[width=1.02\textwidth]{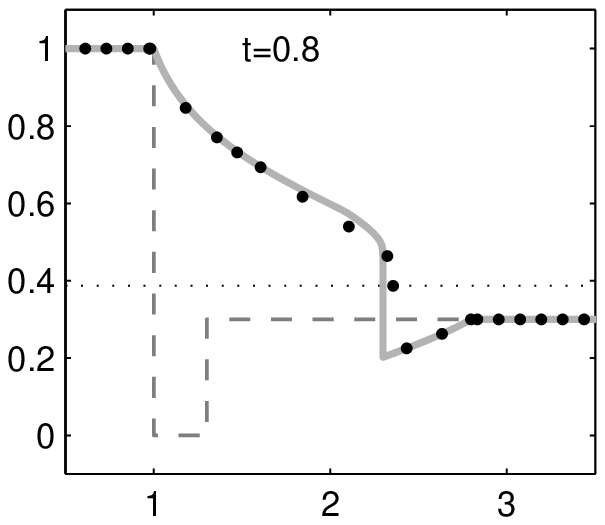}
\end{minipage}
\caption{The particle method solution for the Buckley-Leverett equation at three
different times. The exact solution is given in gray, and the dots show the numerical
approximation.}
\label{fig:results_inflection}
\end{figure}

\subsection{A Simple Source Term}
\label{subsec:numerics_source_terms}
We consider Burgers' equation with a source
\begin{equation*}
u_t+\prn{\tfrac{1}{2}u^2}_x = b'(x)u\;, \qquad b(x) =
\begin{cases}
\cos(\pi x) & x\in [4.5,5.5] \\
0 & \text{otherwise}\;.
\end{cases}
\end{equation*}
It is a simple model for shallow water flow over a bottom profile $b(x)$.
As in \cite{KarlsenMishraRisebro2008}, we consider the domain $x\in [0,10]$.
We include the source into the method of characteristics, as explained
in Sect.~\ref{subsec:source_terms}. Figure~\ref{fig:results_source} shows the numerical
solution for the initial condition 
\begin{equation}
u_0(x) = \begin{cases}
2 & x\le1\\ 
1 & x>1\;.
\end{cases} 
\end{equation}
The particle method (dots) approximates the exact solution (grey line).
A particular aspect in favor of the characteristic approach is the precise (up to the
precision of the Runge-Kutta scheme) recovery of the function values after the obstacle.

\begin{figure}
\centering
\begin{minipage}[t]{.32\textwidth}
\centering
\includegraphics[width=1.02\textwidth]{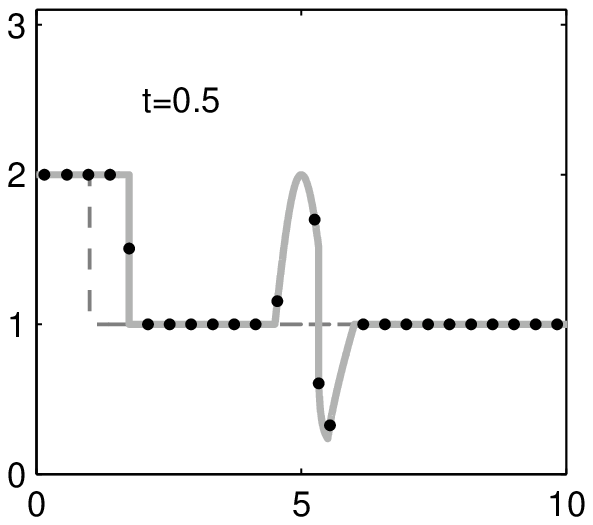}
\end{minipage}
\hfill
\begin{minipage}[t]{.32\textwidth}
\centering
\includegraphics[width=1.02\textwidth]{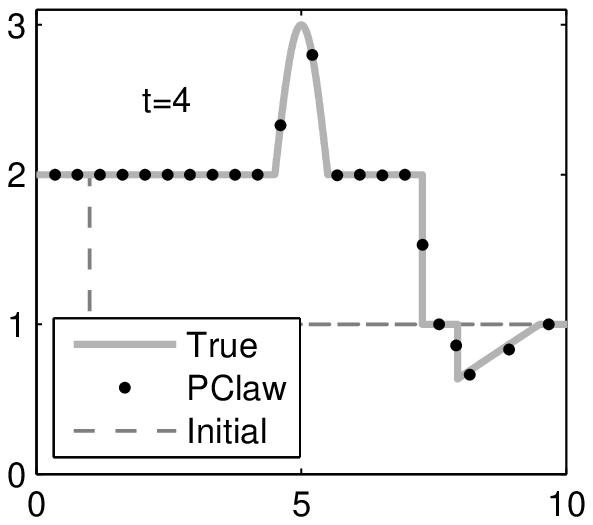}
\end{minipage}
\hfill
\begin{minipage}[t]{.32\textwidth}
\centering
\includegraphics[width=1.02\textwidth]{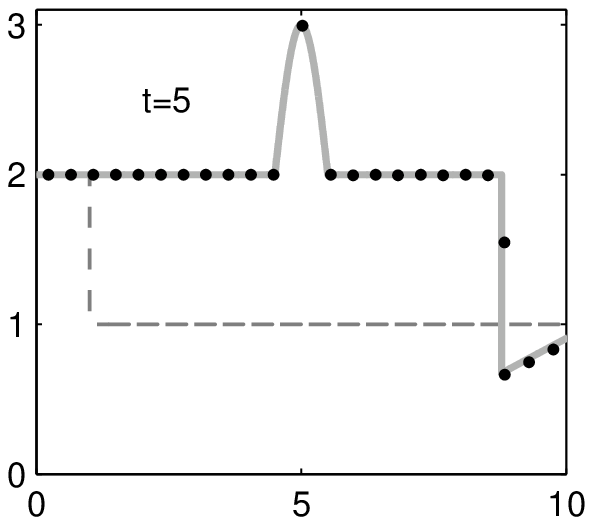}
\end{minipage}
\caption{Numerical results for Burgers' equation with a source. We can see that as time
progresses, a bump is formed, a shock/rarefaction structure emerges,  the wave passes
over the ``bump'', and the two shocks merge into one.}
\label{fig:results_source}
\end{figure}

\section{Further Applications and Extensions}
\label{sec:further_applications}
The presented numerical method employs the method of characteristics
and resolves the interaction of characteristic particles. 
This implies important advantages: the scheme is exact away
from shocks, the provided interpolation yields a certain level of ``subgrid''
resolution, shocks are located very accurately, and the method is TVD. 
A crucial drawback is that the basic method is limited
to scalar problems in one space dimension. 
However, this restriction is not as limiting as it may seem at first glance. 
Many complex problems that arise in applications consist of, or can be
reduced to, scalar one-dimensional sub-problems. 
As we motivate in the following, the presented particle method is a
good candidate for the numerical approximation of such sub-problems.

\subsection{Flow on Networks}
In many models of flows on networks, the evolution on each edge is governed by a scalar
one-dimensional conservation law, with appropriate coupling conditions on the network
nodes. An important example is the simulation of vehicular traffic on road networks. 
Optimization of traffic flow is a topic of current research \cite{HertyKlar2003}, and
frequently a key challenge in such simulations is simply the size of the network. 
For large networks, only a few number of unknowns can be attributed to
each individual road segment. 
Yet, exact conservation properties and accurate shock location are
required in order to track traffic jams. 
In addition, traffic densities must never become negative, hence TVD methods are desirable.
The presented particle method satisfies all those requirements: the method is TVD, the
solution is represented accurately with just a few particles, and shocks are located
accurately.

A simple traffic model that is commonly considered for network flows is
the Lighthill-Whitham model \cite{LighthillWhitham1955}. On each road segment, the
evolution of the vehicle density $\rho(x,t)$ is given by the scalar conservation law
\begin{equation}
\rho_t+\brk{\rho v(\rho)}_x =0\;.
\end{equation}
Here $v(\rho)$ is the speed at which vehicles at density $\rho$ move. It is a decaying
function. A popular choice is
$v(\rho) = v_\text{max}\brk{1-\frac\rho{\rho_\text{max}}}$, where
$v_\text{max}$ is the speed limit assumed on an empty road, and $\rho_\text{max}$
denotes the maximal density of vehicles on the road. This choice of $v(\rho)$ yields a
quadratic flux function, for which the presented particle method is particularly
efficient (see Rem.~\ref{rem:quadratic_flux}).
Another common choice is
$v(\rho) = v_\text{max}\exp\prn{-\frac{\rho}{\rho_0}}$, which results in a
flux function that has one inflection point $\rho^* = 2\rho_0$, and thus
can be treated as described in Sect.~\ref{subsec:inflection_points}.

For a road network, the road properties ($v_\text{max}$, $\rho_\text{max}$,
$\rho_0$) can differ from segment to segment. A proper definition of weak solutions
on such networks has been presented
in \cite{HoldenRisebro1995,CocliteGaravelloPiccoli2005}, by introducing
``demand'' and ``supply'' conditions that have to be satisfied at the nodes.
The translation of these conditions to the particle method is the subject of
current research.

\subsection{Stiff Source Terms}
In Sect.~\ref{subsec:source_terms}, we have considered the case of simple source terms
that are described by the method of characteristics, and that are dominated by the
advection, so that particle management is well described by the source-free similarity
interpolation. Here, we outline a case in which the interpolation between particles can
be used to deal with a dominant source term.

Many examples in chemical reaction kinetics can be described by convection-reaction
equations. Such problems are often times governed by a separation of time scales.
If the advection is scaled to happen on a time scale of $O(1)$, the reactions typically
happen on a much faster time scale $O(\tau)$, where $\tau\ll 1$.
A typical example is the balance law
\begin{equation}
u_t+\prn{f(u)}_x = \psi(u)\;,
\label{eq:convection_reaction}
\end{equation}
where $0\le u\le 1$ represents the density of some chemical quantity. The quantity is
evolved according to a nonlinear advection term, given by a strictly convex flux
function $f$, and a source term. Here we consider a bistable source
$\psi(u) = \tfrac{1}{\tau}u(u-1)(u-\beta)$, where $0<\beta<1$. This source term drives
the values of $u$ towards 0 if smaller than $\beta$, and towards 1 if larger than $\beta$.
This example is presented for instance in \cite{HelzelLevequeWarnecke2000}.
Equation \eqref{eq:convection_reaction} possesses shock solutions as the classical
Burgers' equation does. In addition, it has traveling wave solutions that connect
a left state $u_L \approx 0$ with a right state $u_R \approx 1$. The smooth function that
connects $u_L$ with $u_R$ is of width $\tau$, and it travels with velocity $f'(\beta)$,
as shown in \cite{FanJinTeng2000}.
The speed of propagation of such solutions is due to the interaction of the advection
and the reaction term, and many classical numerical approaches require the resolution
of the wave front in order to obtain the correct propagation speed. This can require an
unnecessarily fine local mesh resolution of $O(\tau)$.

In contrast, the presented particle method possesses key advantages that can be used
to yield correct propagation speeds and width of the front, without actually having to
resolve the wave structure. For the bistable reaction kinetics, we propose to place and
track particles wherever the interpolation crosses or touches the zeros of $\psi$
(here $0$, $\beta$, and $1$).
In a splitting approach, the particles are first moved according to the method of
characteristics \eqref{eq:method_of_characteristics_source}. Since the source is
dominant, this step neglects the evolution of the function between particles that
connect to the value $u=\beta$. Thus, in a correction step, a horizontal motion is added
to the particles which modifies the area between particles in agreement with the source
term. The change of area due to the source is calculated by using the interpolation
\eqref{eq:interpolation_function}, and change in area is found by \eqref{eq:area}. 
While this approach is an approximation (the interpolation
\eqref{eq:interpolation_function} is not the true solution between particles), it can be
shown that for a Riemann problem between $u_L = 0$ and $u_R = 1$, the particle solution
converges to a front that moves at speed $f'(\beta)$ and has a width $O(\tau)$.
The precise treatment of the transient phase when starting with general initial data
shall be addressed in future work.

This approach represents the thin reaction front by exactly three particles.
Other strategies to treat the stiff source term, which give more control over the number
of particles on the front, are subject to current research.

\subsection{Operator Source Terms}
\label{subsec:source_intergral}
Here we consider a conservation law with a source term \eqref{eq:conservation_source},
in which the source $\S u$ may involve derivatives or integrals of the solution.
In these cases, the method of characteristics does not apply. Instead, they can be
treated by fractional steps. In each time step, first the conservation law
$u_t+f(u)_x = 0$ is solved, using the basic particle method described before. Then, in
a second step, the equation $u_t = \S u$ is solved. In many cases, the ``subgrid''
resolution provided by the interpolation \eqref{eq:interpolation_function} between
particles can be employed in this source evolution step, either to improve accuracy
or to treat the evolution analytically.

As an example, consider Burgers' equation with a convolution source term
\begin{equation}
u_t + \prn{\tfrac{1}{2}u^2}_x = K\star u\;, \qquad\text{(with periodic boundary conditions)}
\label{eq:burgers_global_source}
\end{equation}
which can serve as a model for weakly non-linear waves in a closed, nonuniform,
slender box \cite{ShefterRosales1999}. The kernel $K$ embodies the non-uniformity in
the spatial domain.
For a specific family of kernels $K$, the mass ($\int u \ud x$) is always conserved,
and in addition, the total entropy (here $\int u^2 \ud x$) is conserved as well, if the
solution is continuous. The entropy is only diminished at shocks. These properties
are exactly reproduced by our particle method. In contrast, many other numerical
methods can diminish the entropy even in the absence of shocks.

The set of long-term solutions of \eqref{eq:burgers_global_source} has been conjectured
to possess an interesting structure. Using our numerical method we are able to examine
the solution space for evidence of this structure.

\subsection{Higher Space Dimensions}
\label{subsec:higher_dimensions}
A scalar conservation law in higher space dimensions
\begin{equation*}
u_t+\nabla\cdot\vec{f}(u) = 0\;,
\end{equation*}
with a vector-valued flux function $\vec{f}(u) = \prn{f_1(u),\dots,f_d(u)}$,
can be approximated using dimensional splitting. Each time step consists of the successive
solution of one-dimensional conservation laws
\begin{equation}
u_t+\prn{f_i(u)}_{x_i} = 0\;,
\label{eq:conservation_law_di}
\end{equation}
for $i=1,\dots,d$.
While this approach is straightforward for Eulerian methods on regular grids,
Lagrangian approaches have to remesh the solution obtained by
\eqref{eq:conservation_law_di}, before solving in the direction of the next unknown.
In \cite{HoldenRisebro2002}, this approach is outlined for front tracking.
We propose to use an analogous approach for the presented particle method, thus
replacing the tracking of shocks by the tracking of similarity waves.

While higher space dimensions can be resolved by dimensional splitting, this approach
is not fully satisfying, since the required remeshing step spoils one of the main
benefits of the rarefaction tracking approach: the evolution of an exact solution
away from shocks. Similarly as for source terms, it may be more desirable to use the
method of characteristics directly. The characteristic system
\begin{equation*}
\begin{cases}
\dot{\vec{x}} = \nabla \vec f(u) \\
\dot u = 0
\end{cases}
\end{equation*}
yields, as in the one dimensional case, the exact solution on the moving particles.
However, two fundamental complications arise. 
First, at a shock, particles may not collide, but instead move ``around'' each other. 
Second, the definition of an interpolation between particles is more difficult. 
Both aspects are subject to current research.

\section{Conclusions and Outlook}
We have presented a numerical method for one-dimensional scalar conservation laws that
combines the method of characteristics, local similarity solutions, and particle
management. At any time, the solution is approximated by rarefaction and compression
waves, which are evolved exactly. Breaking waves, and thus shocks, are resolved by
local merging of particles. The method is designed to conserve area exactly, not increase
the total variation of the solution, and decrease its entropy. Hence, it is exactly
conservative, TVD, and possesses no numerical dissipation away from shocks. Shocks can
be located accurately by a simple postprocessing step. The method is extended to
non-convex flux functions and to simple balance laws. Numerical solutions for convex
flux functions are second order accurate.

The approach performs well in various examples. Just a few particles yield good
approximations with sharp and well-located shocks. This makes the method attractive
whenever conservation or balance laws arise as subproblems in a large computation, and
only a few degrees of freedom can be devoted to the numerical solution of a single
sub-problem.
The exact conservation of area, and the exact conservation of entropy in
the absence of shocks, make the method a natural candidate for applications
in which conservation of these quantities is crucial.
We have outlined applications in which these features can make the presented particle
method advantageous over classical fixed grid approaches.

Similarly to classical Eulerian and Lagrangian approaches, the presented method can
be extended to more general equations and higher space dimensions using fractional
steps. The idea is to treat advection by particles, and approximate the solution by
similarity waves. Thus, classical 1D Riemann solvers can be replaced by
1D \emph{wave solvers}. However, fractional step methods incur various disadvantages.
A fundamental question is how more general equations and higher space dimensions
can be treated directly.

\section*{Acknowledgments}
The authors would like to thank R.~LeVeque for helpful comments and suggestions.
The support by the National Science Foundation is acknowledged.
Y.~Farjoun was partially supported by grant DMS--0703937.
B.~Seibold was partially supported by grant DMS--0813648.

\bibliographystyle{amsplain}
\bibliography{references_complete}

\end{document}